\documentclass[12pt]{amsart}
\usepackage{geometry}
\usepackage{amssymb,latexsym}
\newcommand{\RNum}[1]{\lowercase\expandafter{\romannumeral #1\relax}}
\usepackage{xcolor}
\usepackage{makecell}
\usepackage{amsthm}
\usepackage{amsmath}
\usepackage{hyperref}
\usepackage{mathrsfs}
\hypersetup{colorlinks=true,linkcolor=blue,urlcolor=blue,citecolor=red}
\newtheorem{theorem}{Theorem}[section]

\newtheorem{proposition}{Proposition}[section]

\newcommand{\beq}{\begin{equation}}
\newcommand{\eeq}{\end{equation}}

\newtheorem{definition}{Definition}[section]
\newtheorem*{equiv*}{Equivalence Relation}

\newtheorem{corollary}{Corollary}[section]
\newtheorem{equivalence}{Equivalence Relation}

\title[EXTENDED EQUIVALENCE OF FUZZY SETS ]{EXTENDED EQUIVALENCE OF FUZZY SETS }

\author{Venkat Murali}
\address{Department of Mathematics, Rhodes University, Grahamstown, 6139, South Africa}
\email{ v.murali@ru.ac.za}

\author{Sithembele Nkonkobe}
\address{Department of Mathematical Sciences, Sol Plaatje University, Kimberley, South\\\indent\indent 8301 Africa}
\email{snkonkobe@gmail.com}

\date{\today}

\begin{document}
	\begin{abstract}  
	  	Preferential equality is an equivalence relation on fuzzy subsets of finite sets and is a generalization of classical equality of subsets. In this paper we introduce a tightened version of the preferential equality on fuzzy subsets   and derive some important combinatorial formulae for the number of such tight  fuzzy subsets of an $n$-element set where $n$ is a natural number. We also offer some asymptotic results.
	\end{abstract}
	\maketitle
	
	\section{Introduction}
A fuzzy subset of an $n$-element set $X_n=\{x_1,x_2,\ldots,x_n\}$ is a mapping $\mu:X_n\rightarrow [0,1]$ (see~\cite{zadehfuzzy}). For $\alpha \in [0,1]$, we refer to the subset  $ \{ x \in X_n :\mu(x) \geq \alpha \} $ of $X_n$ as an $ \alpha-cut$ of $\mu$. Clearly the $\alpha-cuts $ form a chain of crisp subsets of $X_n$ as $\alpha $ runs through [0,1], with order reversing inclusion. The $\mu $ may be captured through $\alpha$-cuts in the following convenient and suggestive way. First we call a maximal chain $\mathcal{C}_{\mu}:X_0\subseteq X_1\subseteq X_2\subseteq\cdots\subseteq X_n$ in the lattice of the power set $\mathcal{P}(X_n)$ a $ flag $ on $X_n$ and secondly we call  a set of ordered real numbers $\lambda_i$  in the unit interval as in    $ \ell_{\mu}:=1=\lambda_0\geq\lambda_1\geq\cdots\geq\lambda_ n\geq 0$ a $ keychain $, where the $\lambda_i$'s are called \textit{pins}. Now by a well-known $\alpha-cuts $ representation (for instance, see~\cite{Murali2006(2),murali2004fuzzy,murali2005(3),MuraliMakamba2005(1)}) we arrive at a fuzzy subset $\mu=\{\lambda_i\chi_{X_i}:0\leq i\leq n\}$ which can be written suggestively as follows: 
$X^1_0\subseteq X_1^{\lambda_1}\subseteq\cdots\subseteq X_n^{\lambda_n}=\{\lambda_i\chi_{X_i}:0\leq i\leq n\}=\mu$. Then, the pair $(\mathcal{C}_{\mu},\ell_{\mu})$ is called a
$pinned \, flag$, having $\mu$ as the associated fuzzy set. Clearly the $\alpha-cuts $ of $\mu$ as above form the chain $\mathcal{C}_{\mu}$.

\begin{equivalence}\cite{rosenfeld1971fuzzy}\label{equivalence:1*}
	Let $\mathcal{G}$ be a group. Then, a mapping $\mu:\mathcal{G}\rightarrow[0,1]$ is said to be a fuzzy subgroup of $\mathcal{G}$ if 
	
	\RNum{1}). $\mu(xy)\geq \min(\mu(x),\mu(y))$ $\forall$ $x,y\in \mathcal{G}$,
	
	\RNum{2}). $\mu(x)=\mu(x^{-1})$ $\forall$ $x\in \mathcal{G}$.
\end{equivalence}

As an equivalence on fuzzy groups Volf in \cite{volf2004counting} studied the following  equivalence relation.
\begin{equivalence}\label{definition:1a}	The fuzzy group $\mu$ is equivalent to another fuzzy group $\nu$ ($\mu\sim\nu$) if for all $x,y\in  X_n$, $\mu(x)>\mu(y)$ if and only if $\nu(x)>\nu(y)$.
	\end{equivalence}
 Equivalence relation \ref{definition:1a} was also used by T{\u{a}}rn{\u{a}}uceanu and Bentea in \cite{tuarnuauceanu2008number} in classification of fuzzy subgroups of finite abelian groups. Makamba in \cite{Makamba1992} also studied the following equivalence relation as an equivalence relation on fuzzy subgroups. 

\begin{equivalence}\label{definition:1*}
	
	The fuzzy group $\mu$ is equivalent to another fuzzy group $\nu$ ($\mu\sim\nu$) if for all $x,y\in  X_n$, 
	
\noindent\RNum{1}).  $\mu(x)>\mu(y)$ if and only if $\nu(x)>\nu(y)$,
	\\\RNum{2}). $\mu(x)=0$ if and only if $\nu(x)=0$.\end{equivalence}

 Subsequently the first author in \cite{murali2005(3)} studied the following equivalence relation on fuzzy sets.

	\begin{equivalence}\label{definition:1}
	
	The fuzzy set $\mu$ is equivalent to another fuzzy subset $\nu$ ($\mu\sim\nu$) if for all $x,y\in  X_n$, 
	\\	\noindent \RNum{1}). $\mu(x)=1$ iff $\nu(x)=1$,
	\\	\noindent\RNum{2}).  $\mu(x)\geq\mu(y)$ if and only if $\nu(x)\geq\nu(y)$,
	\\\RNum{3}). $\mu(x)=0$ if and only if $\nu(x)=0$.	\end{equivalence} 
  
Equivalence relation~\ref{definition:1} was subsequently used by the first author and Makamba in \cite{MuraliMakamba2005(1)} as an equivalence relation on pinned flags. The first author in \cite{Murali2006(2)} used the equivalence relation as an equivalence relation on keychains. In  \cite{kannan2019counting} Kannan and Mohapatra used the equivalence relation as an equivalence relation on fuzzy matrices. Most recently in   \cite{mohapatra2022number} Mohapatra and Hong used the equivalence relation as an equivalence relation on $k$-level fuzzy sets. One notes that viewing the three equivalence relations \ref{definition:1a}, \ref{definition:1*}, and \ref{definition:1} as equivalence relations on fuzzy sets, equivalence relation~\ref{definition:1*} is stronger than equivalence relation~\ref{definition:1a}, and equivalence relation~\ref{definition:1} is stronger than equivalence relation~\ref{definition:1*}. 
The main contribution of this paper is that we propose the following new equivalence relation on fuzzy sets which is stronger than equivalence relation~\ref{definition:1}.

\begin{equiv*}\label{definition:*}For fixed $\{a_0,a_1,\ldots,a_{k+1}\}\in[0,1]$ where $a_0=0$, $a_{k+1}=1$, also  $a_0<a_1<a_2<a_3<\cdots<a_k<a_{k+1}$. Considering the fuzzy sets $\mu:X_n\rightarrow[0,1]\ \backslash\{a_1,\ldots,a_{k}\}$.

	We say $\mu\sim\nu$  if for all $i=0, 1,2,\ldots,k$, and for all $x,y\in X_n$:
	\\\RNum{1}). $\mu(x)=a_{k+1}$ iff $\nu(x)=a_{k+1}$,
	\\\RNum{2}). $a_{i+1}>\mu(x)\geq\mu(y)>a_i$ iff $a_{i+1}>\nu(x)\geq\nu(y)>a_i$,
	\\\RNum{3}). $\mu(x)=a_0$ iff $\nu(x)=a_0$.\end{equiv*}

 In the process of discussing this new equivalence relation we define a new type of keychains and introduce what we name as  $ {\, tight \,\, keychains \,} $,  and their associated pinned flags leading to tight preferential equality and tight fuzzy subsets. Furthermore, we describe tight keychains and kernels of fuzzy subsets and derive some useful properties to apply in the present context.  For the special case $k=0$, the new equivalence relation that we define in this paper becomes equivalence relation~\ref{definition:1}.  Together with interval valued fuzzy subsets, it is useful to study fuzzy subset with membership values in subdivided unit interval by introducing certain fixed numbers, called spikes, in the unit intervals. We require fuzzy subsets to avoid these spikes a membership values. This is akin to a kind of second order fuzzy subsets since membership values restricted to be  in the sub-intervals by a chosen order. For instance when applications received for a number of positions in an institution, some applicants may be termed more suitable than others in the same category. Thus a tightened version of preferential fuzzy subset requires a careful study. It turns out that for all values of $k$ in positive integers the integer sequences for the number of inequivalent pinned flags/inequivalent fuzzy subsets under this new equivalence relation is a well known integer sequence in the literature. It is the same integer sequences for the number of barred preferential arrangements of $X_n$. For barred preferential arrangements see \cite{barred2013,Nkonkobethesis,nkonkobe2020combinatorial,pippenger2010hypercube}.



\section{Image, kernel and ordered partition\label{section:2}}

Suppose $\mu$ is a fuzzy subset of $X_n$. Treating $\mu$ as a function, we follow the usual meaning of image and kernel of a function as applied to $\mu$. Hence, $Im(\mu ) = \{\mu(x) \in [0,1]: x \in X_n\}$ for the $image$  and $ker(\mu) =  \{X_i \subset X_n: i\in Im(\mu)\} $ for the $kernel$ where $X_i =\{x\in X_n: \mu(x) = i, \, i \in Im(\mu)\} $. Then $ker(\mu)$, the kernel of $\mu$ is an ordered partition by virtue of elements in $Im(\mu )$ in the unit interval [0,1]. 	
	So, elements $x_t$ and $x_j$ both belong to $\varpi_i$ if and only if  $\mu(x_t)=\mu(x_j)$. That is, if $Im(\mu)=\{\alpha_i\in I:i=1,2,\ldots f\}$ such that $1\geq\alpha_1>\alpha_2>\cdots>\alpha_f\geq 0$ then $ker(\mu)=\Pi=\{\varpi_1,\varpi_2,\ldots,\varpi_f\}$ where $\varpi_i=\{x\in X_n:\mu(x)=\alpha_i\}$ for $1\leq i\leq f$ i.e. $\mu$=$\bigvee\{\alpha_i\chi_{\varpi_i}:1\leq i\leq f\}$.  Murali in \cite{murali2005(3)} used the same definition for a kernel of a fuzzy set. Gian-Carlo Rota in \cite{rota1964number} also used the same definition for a kernel of a mapping from one set to another. By core of a fuzzy set $\mu$ we mean $core(\mu)=\{x\in X_n:\mu(x)=1\}$. Also, by support of a fuzzy set $\mu$ we mean $supp(\mu)=\{x\in X_n:\mu(x)>0\}$.

\section{THE EXTENDED EQUIVALENCE}

We fix $k$ real numbers ($a_i$'s) in the  interval (0,1) such that $0<a_1<a_2<\cdots<a_k<1$. 
We refer to the nonzero $a_i$'s that are not equal to 1 as \textit{spikes}. A keychain interspersed with spikes is called a $tight \,\, keychain$, in particular a $k$-tight keychain if there are $k$ spikes. Let $I=[0,1]\ \backslash\{a_1,a_2,\ldots,a_k\}$. We denote the set of all fuzzy subsets of $X_n$ by $I^ {X_n}$, where elements of $I^ {X_n}$ are of the form $\mu:X_n\rightarrow[0,1]\ \backslash\{a_1,\ldots,a_{k}\}$.  We refer to the open interval $(a_i,a_{i+1})$ as the $(i+1)^{th}$ interval for all $0\leq i\leq k$. 
\begin{equivalence}[Main Result]\label{definition:2}For fixed $\{a_0,a_1,\ldots,a_{k+1}\}\in[0,1]$ where $a_0=0$, $a_{k+1}=1$, and  $a_0<a_1<a_2<a_3<\cdots<a_k<a_{k+1}$. Fuzzy subsets $\mu,\nu\in I^{X_n}$ are equivalent if and only if for all $i=0, 1,2,\ldots,k$, and for all $x,y\in X_n$:
\\\RNum{1}). $\mu(x)=a_{k+1}$ if and only if $\nu(x)=a_{k+1}$,
 \\\RNum{2}). $a_{i+1}>\mu(x)\geq\mu(y)>a_i$ if and only if $a_{i+1}>\nu(x)\geq\nu(y)>a_i$,
\\\RNum{3}). $\mu(x)=a_0$ if and only if $\nu(x)=a_0$.\end{equivalence}

We denote by $F_n(k)$ the number of inequivalent fuzzy subsets of $X_n$ under equivalence relation~\ref{definition:2}.  Note that for the case $k=0$, equivalence relation ~\ref{definition:2} becomes equivalence relation~\ref{definition:1}. From now on we will refer to equivalence relation ~\ref{definition:2} as the extended equivalence. 

 \begin{definition}[Tight-$\alpha$-cut] we define the tight\\ $\alpha$-cut $\mu^{(R_j>R_i)}=\{x\in X_n: R_j>\mu(x)>R_i \}$.
\end{definition}
From now on when we talk about an $\alpha$-cut we mean a tight $\alpha$-cut.
   The case $k=0$ in equivalence relation \ref{definition:2} has been studied.  We now discuss the case $k>0$. There are three cases to consider. Based on the extended equivalence relation~\ref{definition:2} for each fuzzy subset the spikes induce disjoint intervals, and on  intervals the membership values form a keychain.  We denote by $\alpha_1$ the largest pin of the keychain, and by $\alpha_r$ the smallest pin. We now consider the different cases.

\RNum{1}). Membership values in the interval $(a_k,1]$, and $\alpha_1=1=a_{k+1}$. In this case the keychains are of the form $\ell:a_{k+1}=1=\alpha_1>\alpha_2>\cdots>\alpha_r>a_k$, and the tight $\alpha$-cuts are of the form $\{x\in X_n: 1\geq\mu(x)>a_k \}$. Then the $\alpha$-cut of $\mu$ is equal to $\varpi_1$ for all $1\geq \alpha >\alpha_2$, and is equal to $\cup _{j=1}^t\varpi_j$ for $\alpha_t\geq \alpha>\alpha_{t+1}$, where $2\leq t\leq r$ and $\alpha_{r+1}=a_k$.  

\RNum{2}).  Membership values are in an interval of the form $(a_{i},a_{i+1})$. Then the keychains are of the form $\ell:a_{i+1}>\alpha_1>\alpha_2>\cdots>\alpha_r>a_i$, and the tight $\alpha$-cuts are of the form $\{x\in X_n: a_{i+1}>\mu(x)>a_i \}$. Then the $\alpha$-cut of $\mu$ is $\varpi_1$ for all $a_{i+1}>\alpha>\alpha_2$, and $\cup _{j=1}^t\varpi_j$ for all $\alpha_t\geq \alpha>\alpha_{t+1}$ for all $2\leq t\leq r$,  where $\alpha_{r+1}=a_i$.

\RNum{3}).  Membership values in the interval $[0,a_{1})$ and $\alpha_r=0$. Then the keychains are of the form $\ell:a_{1}>\alpha_1>\alpha_2>\cdots>\alpha_r=0$, and the tight $\alpha$-cuts are of the form $\{x\in X_n: a_{1}>\mu(x)\geq 0 \}$. Then the $\alpha$-cut of $\mu$ is equal to $\varpi_1$ for all $a_{1}>\alpha>\alpha_2$, and $\cup _{j=1}^t\varpi_j$ for all $\alpha_t\geq \alpha>\alpha_{t+1}$, where $2\leq t\leq r-1$.

The argument used in the three cases above is a generalization of the one given by the first author in \cite{murali2005(3)}. In the three cases above for a fixed fuzzy set the same $\varpi_i's$ discussed in different cases to be understood as meaning different subsets of $X_n$ as the membership values are in different intervals. From the cases above it is clear that for membership values of a fuzzy subset $\mu$ that are within the open interval created by two consecutive spikes, we have $ker(\mu)=\Pi$. One note that the same $\Pi$ can be a kernel of several fuzzy subsets. However, the extended equivalence relation takes care of that. We denote by $F_{n,2}(k)$ (and their set by $\mathcal{F}_{n,2}(k)$) the number of inequivalent fuzzy subsets under the extended equivalence relation, where $k$ is the number of spikes ($a_i's$) in the unit interval exclusively between 1 and 0 in accordance with equivalence relation~\ref{definition:2}. We denote by  $F_{n,1}(k)$ (and their set by $\mathcal{F}_{n,1}(k)$) the number of equivalence classes of fuzzy subsets on an interval between a spike and one of the boundary points (either 1 or 0 is the boundary point) i.e. the membership values are in the interval $[0,a_1)$  or are in the interval $(a_k,1]$, the $a_i's$ are spikes. Denote by $F_{n,0}(k)$ (and their set by $\mathcal{F}_{n,0}(k)$) the number of inequivalent fuzzy subsets of $X_n$ under the generalized  equivalence in equivalence relation~\ref{definition:2} where memberships values exclude both boundary points meaning exclude both 1 and 0, that is being in an interval of the form  $(a_i,a_{i+1})$. In this case the number of inequivalent fuzzy sets is the same as the number of ordered partitions as previously noted by Murali in \cite{Murali2006(3)}. It turns out that when considering all the open $k+1$ intervals created by the $k$ spikes the number of inequivalent fuzzy subsets in this special case under the extended equivalence relation is equal to number of barred preferential arrangements studied by the second author in chapter~2 of \cite{Nkonkobethesis}, and studied by Ahlbach et al in \cite{barred2013}. Where the spikes may be interpreted as the bars. We arrive at this fact using the following argument. For each open  interval of the form  $(a_i,a_{i+1})$ on the unit interval there are $r!$ equivalence classes of fuzzy subsets which are associated with the same partition where $r$ is the number of blocks in kernel. So, one can think of the kernel of fuzzy subset of $X_n$ having $r$ blocks as partition of $X_n$, hence, is enumerated by the stirling numbers of the second kind (denoted by $S(n,r)$) (see~\cite{murali2005(3)}).  Furthermore, in accordance with equivalence relation~\ref{definition:2}  taking into account all the $k$ spikes, since all the elements of $X_n$ can have membership values on the open intervals in-between any of the $k$ spikes, and before the first spike and after the $kth$ spike and some of the open intervals between consecutive spikes may be empty, then there are $\binom{k+r}{r}r!$ equivalence classes of fuzzy subsets associated with the same kernel. Thus, \begin{equation}\label{equation:6}F_{n,0}(k)=\sum\limits_{r=0}^{n}\binom{k+r}{r}r!S(n,r),\end{equation}
this is in agreement with Theorem~3 of 
\cite{barred2013}.

\begin{theorem}\label{theoreom:3}For $n\geq 1$ and $k\geq1$:
\begin{equation}\label{equation:5}
	F_{n,2}(k)=4 F_{n,0}(k)-3F_{n,0}(k-1)-\sum\limits_{i=1}^{n}\binom{n}{i}F_{n-i,0}(k-1).
\end{equation}

\end{theorem}
\begin{proof} We construct the inequivalent fuzzy sets of $\mathcal{F}_{n,2}(k)$ under the equivalence relation in equivalence relation~\ref{definition:2} by starting with the set $\mathcal{F}_{n,0}(k)$. A property of the elements of $\mathcal{F}_{n,2}(k)$ is that on the interval $(a_k,1]$ the first block (from left to right) has membership value 1 (which is the core of the fuzzy set) or the first block has membership value strictly less than one (the core of the fuzzy sets in empty). Similarly the last block (from left to right) on each element of $\mathcal{F}_{n,2}(k)$  on the interval $[0,a_1)$ the membership value of the last block is 0 (fuzzy set and its support are equal) or the membership value of the last block is strictly greater than 0 (support is a proper subset of fuzzy set).  So, this gives rise to $4|\mathcal{F}_{n,0}(k)|$. However, there is over counting,  $4|\mathcal{F}_{n,0}(k)|>|\mathcal{F}_{n,2}(k)|$. First reason for over counting is that on some of the elements of $\mathcal{F}_{n,0}(k)$ the first interval is empty i.e. none of the blocks have membership value in the interval $(a_k,1]$. So, for such elements it does not make sense to say the first block in the interval $(a_k,1]$ having membership of either 1 or strictly less than 1.  So, how many element of $\mathcal{F}_{n,0}(k)$ are having this property? From definition there are  $|\mathcal{F}_{n,0}(k-1)|$. Similarly, there are $|\mathcal{F}_{n,0}(k-1)|$ elements that are having the property that $[0,a_1)$ is empty.
	
	 The other class of fuzzy subsets that results into over-counting when on elements of $\mathcal{F}_{n,0}(k)$ both the intervals are empty $(0,a_1)$ and $(a_k,1)$. How many elements of $\mathcal{F}_{n,0}(k)$ are having this property? On this case the number of resultant inequivalent fuzzy sets of $\mathcal{F}_{n,2}(k)$ if $(0,a_1)$  is the one which is empty then the membership values of the resultant inequivalent fuzzy subsets would be all strictly greater than zero. The number inequivalent fuzzy subsets would be $F_{n,1}(k-1)$ from definition, where some elements may have membership degree of 1.  We can obtain the same number $F_{n,1}(k-1)$ by looking at what is the membership value of the first block of each fuzzy subset in this scenario. The membership value of the first block of each fuzzy subset is either 1 or a number in the interval $(a_k,1)$. We can choose the first block of $X_n$ in $\binom{n}{i}$ ways. The degrees of membership of the other $n-i$  are on the other open intervals and are all strictly greater than zero, from definition there are $F_{n,0}(k-1)$, one note that the interval $(0,a_1)$ is empty. Summing over $i$ we have $\sum\limits_{i=0}^{n}\binom{n}{i}F_{n-i,0}(k-1)=F_{n,1}(k-1)$. If the interval $(a_k,1)$ instead of $(0,a_1)$, a similar argument can be used to prove that the number of inequivalent fuzzy subsets is still $F_{n,1}(k-1)$ in this case.
Thus, 	$F_{n,2}(k)=4 F_{n,0}(k)-2F_{n,0}(k-1)-\sum\limits_{i=0}^{n}\binom{n}{i}F_{n-i,0}(k-1)$.

\end{proof}

Theorem~\ref{theoreom:3} above is a generalization of Theorem~3.5 of the first author's paper \cite{Murali2006(3)}.
\begin{theorem}For $n\geq0$,
	\begin{equation}\label{equation:7}
		F_{n+1,2}(k)=2F_{n,2}(k)+(k+1)\sum\limits_{i=0}^n\binom{n}{i}F_{n-i,2}(k+1).
	\end{equation}
\end{theorem}
\begin{proof}
	There are two cases to consider. 
	Case 1: $(n+1)$ is on one of the boundary points on the interval i.e the the membership value of the block having $(n+1)$ is either 1 or 0. The total number of equivalence classes of the membership values of the other $n$  elements is $F_{n,2}(k)$. Hence, the number of inequivalent fuzzy sets  values of the $n+1$ elements satisfying the conditions of equivalence relation~\ref{definition:2} in this case is $2F_{n,2}(k)$.
	
		Case 2: $(n+1)$ is on a block having membership value that is in $(0,1)\backslash\{a_1,a_2,\ldots,a_k\}$.
	 In this case there are $(k+1)$ open interval all of the form $(a_i,a_{i+1})$. One of the intervals can be chosen in $\binom{k+1}{1}$ ways. The other elements that are on the same block as $(n+1)$ can be chosen in $\binom{n}{i}$ ways. Since the $n-i$ other elements all belong to blocks that are either before or after the block having $(n+1)$ then in addition to the $k$, $a_i's$ the membership value of the block containing $(n+1)$ may be regarded as the $(k+1)th$ point on the interval $[o,1]$ that any membership value of the formed blocks of the $(n-i)$ other elements may not have, thus the number of possible inequivalent equivalence classes for the $n-i$ elements with this property is $F_{n-i,2}(k+1)$. Thus, the total number of inequivalent fuzzy subsets of $X_{n+1}$ in this case is $(k+1)\sum\limits_{i=0}^n\binom{n}{i}F_{n-i,2}(k+1)$.

\end{proof}

The second author and co authors in a different context on Theorem 2.4 of \cite{nkonkobeetal2020(1)} discussed a generalization of the identity on \eqref{equation:7}.
\begin{theorem}For $n\geq0$,\label{theorem:t}
	\begin{equation}
		F_{n,2}(k)=\sum\limits_{i=0}^n\binom{n}{i}\begin{bmatrix}\sum\limits_{l=0}^il!S(i,l)\binom{k+l}{l}\end{bmatrix}2^{n-i}.
	\end{equation}
	
\end{theorem}
\begin{proof}
	The membership values of the formed blocks of the fuzzy sets of elements of $X_n$ under the equivalence relation in equivalence relation~\ref{definition:2} can be classified as being in one of two places. The membership value of each block is either at a boundary point ( membership value 1 or 0) or at a non-boundary point in the open interval (0,1).    The number of elements whose formed blocks are at boundary points can be chosen in $\binom{n}{i}$. The $i$ elements can occupy the boundary points in $2^i$ ways. By equivalence relation~\ref{equation:6} the remaining $n-i$ elements can be arranged on the open intervals $(0,1)\backslash\{a_1,a_2,\ldots,a_k\}$ in $F_{n-i,0}(k)$ ways.  
\end{proof}

We now compare equivalence relation~\ref{definition:2} to equivalence relation~\ref{definition:1}.  Amongst the equivalence relations we have discussed, we have chosen to compare \ref{definition:2} to \ref{definition:1} because among the three equivalence relations (equivalence relations \ref{definition:1a}, \ref{definition:1*}, and \ref{definition:1})  \ref{definition:1} is the strongest. Without loss of generality we will use $X_1$ in out running example. In \cite{MuraliMakamba2005(1)} Murali and Makamba discussed possible distinct keychains of a given length that can be used in forming pinned flags. For instance, keychains of length 2, there are 3 distinct keychains of length 2 which are $11$, $1\lambda$ (where $0<\lambda<1$), and $10$. These are keychains that can be used to find distinct pinned flags/inequivalent fuzzy subsets of $X_1=\{1\}$. These distinct pinned flags/inequivalent fuzzy subsets are  $X_0^1\subset X_1^1$, $X_0^1\subset X_1^\lambda$, and $X_0^1\subset X_1^0$ based on equivalence relation~\ref{definition:1}. However, based on  equivalence relation~\ref{definition:2} the number of distinct fuzzy subsets depends on the value of $k$ even for a singleton set. For instance, say $k=2$, meaning we have 2 spikes in the unit interval. Say $a_1=0.25$, and $a_2=0.75$. Now, the keychains of length 2 are $11$, $1\lambda^1$, $1\lambda^2$,$1\lambda^3$, and $10$, where $0<\lambda^1<0.25$, $0.25<\lambda^2<0.75$, and $0.75<\lambda^3<1$. So, there are five distinct keychains. Hence, some of the inequivalent fuzzy subsets are of the form $X_0\subset X_1^{\lambda_i}$, the superscript  in the $1\lambda^i$'s refers to position of spikes. In general $X_1$ has $k+3$ distinct fuzzy subsets, where $k$ is the number of spikes, and is a non-negative integer. We have seen that fuzzy subsets based on the three keychains $1\lambda^1$, $1\lambda^2$,$1\lambda^3$ are all distinct based on equivalence relation~\ref{definition:2}, however the same fuzzy subsets are all equivalent based on equivalence relation \ref{definition:1}. In general keychains used in equivalence relation ~\ref{definition:2} are of the form  $\ell:=1=\lambda_0\geq\lambda_1^k\geq\cdots\geq\lambda_{n_k}^k>a_k>\lambda_1^{k-1}\geq\cdots\geq\lambda_{n_{k-1}}^{k-1}>a_{k-1}>\cdots >a_1>\cdots \geq0$, where the position of spikes is indicated.

We now compare number of inequivalent fuzzy subsets Based on  equivalence relation~\ref{definition:2} and equivalence relation~\ref{definition:1} for more values of $n$ in the form of a table.

	\begin{table}[h!]\caption{A comparison of two equivalence relations.}
	\begin{tabular}{|c|c|c|}\hline
		\thead{\makecell{The set}}&\thead{Number of \\inequivalent fuzzy subsets\\using equivalence relation~\ref{definition:1}\\ This is the same\\ integer sequence \\as that of number of chains\\ in the power set\\ of $X_n$(see~\cite{nelsen1991chains}) }&	\thead{Number of\\ inequivalent\\ fuzzy subsets\\using \\ equivalence\\ relation~\ref{definition:2}}  \\ \hline
		$X_1$&\makecell[l]{Based on the keychains\\$11$, $1\lambda$, and $10$, there are \\three inequivalent\\ fuzzy subsets:\\ 1). $X_0^1\subset X_1^1$,\\ 2). $X_0^1\subset X_1^\lambda$, where $0<\lambda<1$\\ 3). $X_0^1\subset X_1^0$}&\makecell[l]{ For argument sake we suppose\\ that $k=3$(number of spikes).\\ Say we have $a_1=0.3$, $a_2=0.5$,\\ and $a_3=0.8$. \\so we have membership values \\in $[0,1]\backslash\{0.3,0.5,0.8\}$\\ The inequivalent fuzzy subsets are
			\\ 1). $X^1_0\subset X_1^1$\\2). $X^1_0\subset X_1^{\lambda_4}$ where $0.8<\lambda_4<1$\\3). $X^1_0\subset X_1^{\lambda_3}$ where $0.5<\lambda_3<0.8$\\4). $X^1_0\subset X_1^{\lambda_2}$ where $0.3<\lambda_2<0.5$
			\\5). $X^1_0\subset X_1^{\lambda_1}$ where $0<\lambda_1<0.3$
			\\6). $X^1_0\subset X_1^0$\\ Hence, there are six inequivalent\\ fuzzy sets for $X_1.$}\\\hline
		$X_2$& \makecell[l]{Using Theorem~\ref{theorem:t} the number \\of inequivalent fuzzy subsets is 11.} &\makecell{Using Theorem~\ref{theorem:t} \\for $k=3$ the number \\of inequivalent fuzzy subsets is 44.}\\\hline
	$X_3$	&\makecell[l]{Using Theorem~\ref{theorem:t} the number \\of inequivalent fuzzy subsets is 51.}&\makecell[l]{Using Theorem~\ref{theorem:t} for $k=3$\\the number of inequivalent \\fuzzy subsets is 384.}\\\hline

	\end{tabular}
\end{table}

\newpage\section{Asymptotics}
In this section, we propose  some asymptotic results for the numbers $F_{n,2}(k)$. The results that we discuss are based on a method developed in the papers \cite{hsu1990power,hsu1991kind}. In \cite{nkonkobe2020combinatorial} this method has been applied to the generating function \begin{equation}\label{equation:300}\frac{(1+\alpha t)^{\gamma/\alpha}}{[1-x[(1+\alpha t)^{\beta/\alpha}-1]]^k}=\sum\limits_{n=0}^\infty T^{k,x}_n(n;\alpha,\beta,\gamma)\frac{t^n}{n!}.\end{equation} The method has also been used in  \cite{adell2023unified,corcino2022second,hsu1998unified}. We have (see~\cite{barred2013}), \begin{equation}\label{equation:8}
	\sum\limits_{n=0}^\infty F_{n,0}(k)\frac{t^n}{n!}=\frac{1}{(2-e^t)^{k+1}}.  
\end{equation} The polynomials in \eqref{equation:8} are a special case of the polynomials studied in \cite{nkonkobe2020combinatorial}. For $n\in\mathbb{N}$,  we suppose $k_1+2k_2+\cdots+nk_n=n$, and $k=k_1+k_2+\cdots+k_n$ is the number of parts of the partition $1^{k_1}2^{k_2}\cdots n^{k_n}$. Let $u\in\mathbb{C}$ and $n\in\mathbb{N}$. Let $(u)_n$ denote the product $u(u-1)(u-2)\cdots(u-n+1)$, where $(u)_0=1$. We also let  $\sigma(n)$ denote the set of partitions of the integer $n\in \mathbb{N}$ and  $\sigma(n,k)$ denote the set of partitions of $n\in \mathbb{N}$ that are having $k$ parts.  We let  $\tau(t)$ denote the formal power series $\sum\limits_{n=0}^\infty a_nt^n$ over $\mathbb{C}$ with $\tau(0)=a_0=1$. We suppose that, for any $\beta\in \mathbb{C}$ with $\beta\not= 0$, we have a formal power series

\begin{equation}
	\eta(t)=(\tau(t))^\beta=\sum_{n=0}^\infty {\beta\brace n}t^n,\;
\end{equation} such that ${ \beta \brace n}=[t^n]\eta(t)$ and ${\beta\brace 0}=1$. Then, for $\epsilon\in \mathbb{C}$ with $\epsilon\not=0$ we have

\begin{equation}\frac{1}{(\epsilon)_n}\left[t^n\right]\left(\eta(t)\right)^{\epsilon}=\frac{1}{(\epsilon)_n}{\beta \epsilon\brace n}=\sum_{j=0}^s\frac{W(n,j)}{(\epsilon-n+j)_j}\;\;+o\left(\frac{W(n,s)}{(\epsilon-n+s)_s}\right),\end{equation}
where $W(n,i)=\sum\limits_{\sigma(n,n-i)}\frac{a_1^{k_1}a_2^{k_2}\cdots a_n^{k_n}}{k_1!k_2!\cdots k_n!}$, $0\leq i<n$,  and $n=o(\sqrt{|\epsilon|})$ as  $|\epsilon|\rightarrow \infty$.

\begin{theorem}\label{theorem:30} For all  $n\geq0$, we have

	\begin{equation}
		\frac{{\beta \epsilon\brace n}}{(\epsilon)_n}=\sum\limits_{j=0}^q\frac{W(n,j)}{(\epsilon-n+j)_j}+o\begin{pmatrix}
			\frac{W(n,q)}{(\epsilon-n+q)_q}
		\end{pmatrix}.
	\end{equation}

\end{theorem}

\begin{corollary}\label{theorem:200} For all  $n\geq0$, we have

	\begin{equation}
		\frac{F_{n,2k+2}(k)}{(\epsilon)_nn!}=\sum\limits_{j=0}^q\frac{W(n,j)}{(\epsilon-n+j)_j}+o\begin{pmatrix}
			\frac{W(n,q)}{(\epsilon-n+q)_q}
		\end{pmatrix},
	\end{equation}
	
	\noindent	where \begin{equation}\label{equation:200}W(n,j)=\sum\limits_{\sigma(n,n-j)}\prod\limits_{i=1}^n\frac{1}{k_i!}\begin{bmatrix}
			\frac{F_{i,2}(k)}{i!}
		\end{bmatrix}^{k_i},\end{equation} and $n=o(\sqrt{|\epsilon|})$ as  $|\epsilon|\rightarrow \infty$.
	
\end{corollary}

We compute a few values of  $W(n,j)'s$. We have:

\begin{align*}
	W(n,0)=&\frac{1}{n!}\begin{Bmatrix}\frac{F_{1,2}(k)}{1!}\end{Bmatrix}^{n},\\
	W(n,1)=&\frac{1}{(n-2)!}\begin{Bmatrix}\frac{F_{1,2}(k)}{1!}\end{Bmatrix}^{n-2}\begin{Bmatrix}\frac{F_{2,2}(k)}{2!}\end{Bmatrix},\\
	W(n,2)=&\frac{1}{(n-3)!}\begin{Bmatrix}\frac{F_{1,2}(k)}{1!}\end{Bmatrix}^{n-3}\begin{Bmatrix}\frac{F_{3,2}(k)}{3!}\end{Bmatrix}\\&+\frac{1}{2!(n-4)!}\begin{Bmatrix}\frac{F_{1,2}(k)}{1!}\end{Bmatrix}^{n-4}\begin{Bmatrix}\frac{F_{2,2}(k)}{2!}\end{Bmatrix}^2,\\
	W(n,3)=&\frac{1}{(n-4)!}\begin{Bmatrix}
		\frac{F_{1,2}(k)}{1!}
	\end{Bmatrix}^{n-4}\begin{Bmatrix}
		\frac{F_{4,2}(k)}{4!}
	\end{Bmatrix}\\&+\frac{1}{(n-5)!}\begin{Bmatrix}
		\frac{F_{1,2}(k)}{1!}
	\end{Bmatrix}^{n-5}\begin{Bmatrix}
		\frac{F_{2,2}(k)}{2!}
	\end{Bmatrix}\begin{Bmatrix}
		\frac{F_{3,2}(k)}{3!}
	\end{Bmatrix}\\&+\frac{1}{3!(n-6)!}\begin{Bmatrix}
		\frac{F_{1,2}(k)}{1!}
	\end{Bmatrix}^{n-6}\begin{Bmatrix}
		\frac{F_{2,2}(k)}{2!}
	\end{Bmatrix}^3.
\end{align*}

Finally, by Theorem~\ref{theorem:30} :

\begin{align*}\frac{F_{n,2k+2}(k)}{n!}\sim\: & (\epsilon)_nW(n,0)+(\epsilon)_{n-1}W(n,1)+(\epsilon)_{n-2}W(n,2)\\&+(\epsilon)_{n-3}W(n,3).\end{align*}
\section{Conclusion and Further Research} 
The main contribution of this paper is proposed a new type of fuzzy sets, where certain fixed values in the unit interval are omitted. We also propose some asymptotic results. We defined an equivalence relation on these fuzzy sets, leading to tight preferential fuzzy sets, and tight-keychains. This equivalence relation is a generalization of a number of equivalence relations on fuzzy sets in the literature. Hence, existing results in the literature on fuzzy sets based on weaker equivalence relations may be extended to equivalence relation ~\ref{definition:2}. For instance, Murali and Makamba in \cite{MuraliMakambaE1} used equivalence relation~\ref{definition:1*} (which is weaker than equivalence relation~\ref{definition:2}) to  prove the following theorem.
\begin{theorem}
	If $\mu$, and $\nu$ are fuzzy subsets of $X_n$ then $|Im(\mu)|=|Im(\nu)|$.
\end{theorem}

It can be shown that the same statement is true also for fuzzy subsets when using equivalence relation~\ref{definition:2}.

Murali and Makamba in \cite{MuraliMakamba2005(1)} used equivalence relation~\ref{definition:1} (which is weaker than equivalence relation~\ref{definition:2}) to prove the following result.
\begin{proposition}\label{proposition:1}Given two pinned flags of two fuzzy sets $\mu$ and $\nu$ 
	$$X^1_0\subset X_1^{\lambda_1}\subset\cdots\subset X_n^{\lambda_n}=(C_\mu,l_\mu), \text{and}$$ 
	$$Y^1_0\subset Y_1^{\gamma_1}\subset\cdots\subset Y_r^{\gamma_r}=(C_\nu,l_\nu).$$ 
	
	Then $\mu\sim \nu$ if and only if 
	
	\RNum{1}. $n=r$,
	
	\RNum{2}. $X_i=Y_i$ for all $0\leq i\leq n$; provided the $\lambda_i's$ and the $\gamma_i's$ are distinct.
	
	\RNum{3}. $\lambda_i>\lambda_j$ if and only if $\gamma_i>\gamma_j$  for $i\geq 1$, $j\leq n$ and $\lambda_k=0$ if and only if $\beta_k=0$ for some $k$ between 1 and $n$.
	
\end{proposition}

Using  equivalence relation~\ref{definition:2} Proposition~\ref{proposition:1} may be extended to incorporate the existence of spikes in the unit interval.
	
	
	
	
	
	
	
One of classical problems in fuzzy group theory is that of classification of fuzzy groups. Equivalence relations have been used in classification of fuzzy groups. For instance in \cite{gideon2014classification,murali2008methods,MuraliMakambaE3,MuraliMakambaE1,MuraliMakambaE2,murali2004counting,tuarnuauceanu2008number,volf2004counting}, where equivalence relations \ref{definition:1a}, and  \ref{definition:1*} have been used. Equivalence relations are important in fuzzy group theory because equivalence relations create conductive setting for classifying subgroups of a given group. For future research the equivalence relation that we propose in this study (equivalence relation~\ref{definition:2}) may be used as an equivalence relation on fuzzy groups.
 In the process the results obtained would generalize several results previously obtained in classification of fuzzy groups. In conducting this research the following adjustment would need to be made. The keychains would need to be of the form:
  
   $\ell:=1=\lambda_0\geq\lambda_1^k\geq\cdots\geq\lambda_{n_k}^k>a_k>\lambda_1^{k-1}\geq\cdots\geq\lambda_{n_{k-1}}^{k-1}>a_{k-1}>\cdots >a_1>\cdots \geq0$, this is to reflect the existence of the spikes in the unit interval. Also, the pinned flags on the lattice of subgroups of a group $G$ would be of the form  $G_0^1\subset G^{\lambda_1^i}_1\subset G^{\lambda_1^j}_2\subset\cdots\subset G^{\lambda_n^r}_n$.


\bibliography{bibtexfile.bib}{}
\bibliographystyle{plain}

\end{document}